\documentclass[a4paper,20pt]{amsproc}

\usepackage{bbm}
\usepackage{hyperref}
\newtheorem{theorem}{Theorem}[section]
\newtheorem{lemma}[theorem]{Lemma}
\newtheorem{corollary}[theorem]{Corollary}
\newtheorem{proposition}[theorem]{Proposition}

\theoremstyle{definition}
\newtheorem{definition}[theorem]{Definition}
\newtheorem{example}[theorem]{Example}
\newtheorem{question}[theorem]{Question}

\theoremstyle{remark}

\newcommand{\R}{\mathbb{R}}

\newcommand{\C}{\mathbb{C}}
\newcommand{\K}{\mathbb{K}}
\newcommand{\lowK}{\mathbbm{k}}

\numberwithin{equation}{section}

\begin{document}

\title[Extension of Local Maps. Applications and Examples]{A New Method of Extension of Local Maps of Banach Spaces. Applications and Examples}


\author[Genrich Belitskii]{Genrich Belitskii}
\address{Department of Mathematics and Computer Science\\ Ben Gurion University of the Negev\\ P.O.B. 653\\ Beer Sheva, 84105\\ Israel}
\email{genrich@cs.bgu.ac.il}
\thanks{The authors thank the referees for helpful suggestions.}

\author[Victoria Rayskin]{Victoria Rayskin}
\address{Department of Mathematics\\ Tufts University\\ Medford, MA 02155-5597}
\email{victoria.rayskin@tufts.edu}


\date{}

\begin{abstract}
A known classical method of extension of smooth local maps of Banach spaces uses smooth bump functions. However, such functions are absent in the majority of infinite-dimensional Banach spaces. This is an obstacle in the development of local analysis, in particular in the questions of extending local maps onto the whole space. We suggest an approach that substitutes bump functions with special maps, which we call blid maps. It allows us to extend smooth local maps from non-smooth spaces, such as $C^q[0,1], q=0,1,...$. As an example of applications, we show how to reconstruct a map from its derivatives at a point, for spaces possessing blid maps. We also show how blid maps can assist in finding global solutions to cohomological equations having linear transformation of argument.

\end{abstract}

\maketitle

\section{Introduction}\label{sec-intro}
With the advancement of dynamical systems and analysis, the complexity of global analysis became evident. This stimulated the development of techniques for the study of local properties of a global problem. One of the methods of localization is based on the functions with bounded support. The history of applications of functions vanishing outside of a bounded set goes back to the works of Sobolev (\cite{S}) on generalized functions.\footnote{Sobolev was a student of N.M. Guenter, and this work was probably influenced by Guenter. However, Guenter was accused in the development of "abstract" science at the time when the USSR was desperate for an applied theory for the creation of atomic weapons. Guenter was forced to resign from his job.} Later,  functions with bounded support were used by Kurt Otto Friedrichs in his paper of 1944. His colleague, Donald Alexander Flanders, suggested the name mollifiers. Friedrichs himself acknowledged Sobolev's work on mollifiers stating that: "These mollifiers were introduced by Sobolev and the author". A special type of mollifier, which is equal to 1 in the area of interest and smoothly vanishes outside of a bigger set, we call a  bump function.

There are many examples, where bump functions are used for the study of local properties of dynamical systems in $\R^n$. For instance, see \cite{N} and \cite{St}. J. Palis in his work \cite{P} considers bump functions in Banach spaces. He proves the existence of Lipschitz-continuous extensions of local maps with the help of Lipschitz-continuous bump functions. However, Z. Nitecki (\cite{N}) points out that generally speaking, the smoothness of these extensions may not be higher than Lipschitz.

This is an obstacle in the local analysis of dynamical systems in infinite-dimensional spaces. The majority of infinite-dimensional Banach spaces do not have smooth bump functions. In the works \cite{B}, \cite{B-R}, \cite{R} we discuss the conditions when two $C^{\infty}$ diffeomorphisms on some Banach spaces are locally $C^{\infty}$-conjugate. To construct the conjugation, we use bounded smooth locally identical maps. We call them blid maps. Blid maps are the maps that substitute bump functions and allow localization of Banach spaces.

The main objectives of this paper are to present blid maps (Section~\ref{sec-blid maps}) and to introduce the questions of existence of smooth blid maps on various infinite dimensional Banach spaces and their subsets (Sections~\ref{sec-examples},~\ref{sec-questions}). One of the important questions that arises in this topic is the following: Which infinite dimensional spaces possess smooth blid maps? 

As an application example, for spaces possessing blid maps, we prove an infinite-dimensional version of the Borel lemma on a reconstruction of $C^\infty$ map from the derivatives at a point (Section~\ref{sec-Borel-lemma}). We also show how blid maps assist in the proof of decomposition lemmas, frequently used in local analysis (Section~\ref{sec-decomp-lemmas}). Finally, in Section~\ref{section-equation} we apply blid maps to the investigation of cohomological equations with a linear transformation of an argument, which frequently arise in the normal forms theory.

We discuss a possibility of extension of a map $f:U\to Y$, where $U$ is a neighborhood of a point $z$ in a space $X$. More precisely, does there exist a mapping defined on the entire $X$, which coincides with $f$ in some (smaller) neighborhood? Because we do not specify this neighborhood, there arises the following notion of a germ at a point. 

Let $X$ be a real Banach space, $Y$ be either a real or complex one, and $z\in X$ be a  point. Two maps $f_1$ and $f_2$ from neighborhoods $U_1$ and $U_2$ of the point $z$ into $Y$ are called {\it equivalent} if there is a neighborhood $V \subset U_1\cap U_2$   of $z$ such that both of the maps coincide on $V$. A {\it germ at $z$} is an equivalence class.
Therefore, every local map $f$ from a neighborhood of $z$ into $Y$ defines a germ at $z$. Sometimes in the literature it is denoted by $[f]$, although in general we will use the same notation, $f$, as for the map.
 
We consider Fr{\'e}chet $C^q$-maps with $q=0,1,2...,\infty$. All notions and notations of differential calculus in Banach spaces we borrow from \cite{C}.
 
 For a given $C^q$-germ $f$ at $z$, we pose the following questions.
 \begin{question}\label{quest-exists-global}
 Does its global $C^q$-representative (i.e., a $C^q$-map defined on the whole $X$) exist?
 \end{question}
\begin{question}\label{quest-exists-global-bounded}
Assume that $f$ has  local representatives with bounded derivatives. Does there exist a global one with the same property?
\end{question}
It was shown in the works \cite{DH} and \cite{HJ} that there exist separable Banach spaces that do not allow $C^2$-smooth extension of a local $C^2$-representative\footnote{We thank reviewers of this paper for the reference to this result.}. 
\par
Below, without loss of generality we assume that $z=0$.
\par
Usually for extension of local maps described in Question~\ref{quest-exists-global} and Question~\ref{quest-exists-global-bounded} bump functions are used. The classical definition of a bump function (\cite{S}) is a non-zero bounded $C^q$-function from $X$ to $\R$ having a bounded support. We use a similar modified definition which is more suitable for our aims. Namely, a bump function at $0$ is a $C^q$-map $\delta_U:X \to \R$  which is equal to $1$ in a neighborhood  of $0$ and vanishing outside of a lager neighborhood $U$. If $f$ is a local representative of a $C^q$-germ defined in a neighborhood $V$ ($\overline{U}\subset V$) then

\begin{equation}\label{eqn-F} 
F(x)=\left\{
\begin{array}{ll}
\delta_U(x)f(x),& x\in U\\
0,& x\notin U
\end{array}
\right.
\end{equation}

is a global $C^q$-representative of the germ $f$, and it solves at least the Question~\ref{quest-exists-global}. If, in addition, all derivatives of $\delta_U$ are bounded on the entire $X$, then \eqref{eqn-F} solves both of the Questions. If these functions do exist for any $U$, then every germ has a global representative.

A continuous bump function exists in any Banach space. It suffices to set 
$\delta(x)=\tau(||x||)$, where $\tau$ is a continuous bump function at zero on the real line.
    Let $p=2n$ be an even integer. Then
                       $$  \delta(x)=\tau(||x||^p)$$
is a $C^\infty$-bump function at zero on $l_p$. Here $\tau$ is a $C^\infty$-bump function on the real line.
 
However, if $p$ is not an even integer, then $l_p$ space does not have $C^q$-smooth ($q>p$) bump functions (see~\cite{M}). The Banach-Mazur theorem states that any real separable Banach space is isometrically isomorphic to a closed subspace of $C[0,1]$.  Consequently, the space of $C[0,1]$ does not have smooth bump functions at all 
(see~\cite{K} or \cite{M}). Following V.Z. Meshkov (\cite{M}), we will say that a space is $C^q$-smooth, if it possesses a $C^q$-bump function.\footnote{Usually, a space is called smooth if it satisfies a similar property related to the smoothness of a norm (see, for example, \cite{F-M}). However, we will adopt the definition of \cite{M}.}
 
\begin{example}\label{example-integral}
The real function
           $$
          f(x)=\int_0^1 \frac{dt}{1-x(t)},\ \    x\in C[0,1]
 $$
defines a $C^\infty$ (which is even analytic) germ at zero. In spite of the absence of smooth bump functions, the germ has a
global $C^\infty$ representative. To show this, let $h$ be  a real $C^\infty$-function on the real line such that
\begin{equation}\label{eqn-h}
          h(s)=\left\{
\begin{array}{ll}
s, &|s|<1/3\\
0, &|s|>1/2. 
\end{array}
\right.
\end{equation}

Then the $C^\infty$-function
 $$
                F(x)=\int_0^1 \frac{dt}{1-h(x(t))}
 $$
coincides with $f$ in the ball $||x||<1/3$ and is a global representative of the germ with bounded derivatives of all orders.
\end{example}

\section{Blid maps}\label{sec-blid maps}
\begin{definition}\label{def-blid map} A  $C^q$-{\it blid map} for a Banach space $X$ is a global {\bf B}ounded {\bf L}ocal {\bf Id}entity at zero $C^q$-map $H:X \to X$.
\end{definition}

     In other words, $H$ is a global representative of  the germ at zero of identity map such that
 $$
                               \sup_x || H(x)||<\infty.
 $$
The existance of blid maps allows locally defined mappings to be extended to the whole space.
 
\begin{theorem}\label{thm-extension} Let a space $X$ possesses a $C^q$-blid map $H$. Then for every Banach space $Y$ and any $C^q$-germ $f$ at zero from $X$ to $Y$ there exists a global $C^q$-representative. Moreover, if all derivatives of $H$ are bounded, and $f$ contains a local representative bounded together with all its derivatives, then it has a global one with the same property.
\end{theorem}

\begin{proof} 
Let $||H(x)||<N$ for all $x\in X$, and $H(x)=x$ for $||x||<n$. Further, let $f$ be a representative of a germ defined on a neighborhood $U$, and let a closed ball $B_{\epsilon}=\{x: ||x|| \leq \epsilon\}\subset U$. The map
\begin{equation}
              H_1(x)=\frac{\epsilon}{N} H\left(\frac{N}{\epsilon} x\right)
\end{equation}
is a $C^q$-blid map also, and its image is contained in $U$. Therefore the map
 \begin{equation}\label{eqn-global-F}       
F(x)=f(H_1(x))
\end{equation}
is well-defined on the whole space $X$, and it coinsides with the map $f$ in the neighborhood  $||x||<\frac{\epsilon n}{N}$.
The map $F$ is a global $C^q$-representative of the germ $f$. If both of the maps $H$ and $f$ are bounded together with all of their derivatives, then $F$ possesses the same property. This completes the proof.
 \end{proof}
  
\section{Examples}\label{sec-examples}
Let us present spaces having blid maps. 
 
\begin{itemize} 
\item[1.]  Let $X$ be $C^q$-smooth, and let $\delta (x)$ be a $C^q$-bump function at zero. Then
 $$
                   H(x)=\delta(x) x
 $$
is a $C^q$-blid map. If the bump function is bounded together with all its derivatives, then $H$ has the same property.
\item[2.]  Let  $X=C(M)$ be the space (a Banach algebra) of all continuous functions on a compact Hausdorff space $M$ with
 $$
                     ||x||=\max_t|x(t)| ,\ \  t\in M,
 $$
and let h be a $C^\infty$-bump function on the real line. Then the map
 
 \begin{equation}\label{C[01]-blid map}
                                     H(x)(t)=h(x(t))x(t),\ \  x\in X
 \end{equation}
is $C^\infty$-blid map with bounded derivatives of all orders. 

Indeed, since $h$ is locally equal to 1, $H$ is a local identity. Let $\epsilon$ be a positive real number such that $h(\tau)\equiv 0$ for all $|\tau| >\epsilon$. We can always find such $\epsilon$, because bump functions have bounded support.\\ 
Then, $$||H(x)(t)|| = || h(x(t))x(t) ||\leq \epsilon.$$ Also, \\
$$
||H'(x)(t)||= ||h'(x(t)) x(t)+ h(x(t))|| \leq \epsilon \sup|h'| +1.
$$
\\
Similarly, one can show boundedness of all higher order derivatives.
\item[3.]  More generally, let $X\subset C(M)$ be a subspace such that
 
\begin{equation}\label{ideal-property}
x(t)\in X \implies h(x(t))x(t)\in X
 \end{equation}
for any $C^\infty$  bump function $h$ on the real line.

 Then~\eqref{C[01]-blid map} defines a $C^\infty$-blid map with bounded derivatives. For example, any ideal $X$ of the algebra satisfies \eqref{ideal-property}.
\item[4.] Let $X=C^n(M)$ be the space (which is also a Banach algebra)  of all $C^n$-functions on a smooth compact manifold $M$ with or without boundary, and
$$
         ||x||=\max_k\max_t ||x^{(k)}(t)||,\  k\leq n,\ t\in M.
 $$
Then \eqref{C[01]-blid map} gives a $C^\infty$-blid map with bounded derivatives. The same holds for any closed subspace $X\subset C^n(M)$ satisfying \eqref{ideal-property}. As above, $X$ may be an ideal of the algebra.
\item[5.]  Let $X$ possess a $C^q$-blid map $H$, and a subspace\footnote{By definition, a subspace of a Banach space is always closed.} $X_1$ of $X$ be $H$-invariant. Then the restriction $H_1=H|X_1$ is a $C^q$-blid map on $X_1$.
\item[6.]  Let $\pi:X \to X$ be a bounded projector and $X$ possess $C^q$-blid map $H$. Then the restriction $\pi(H)|Im(\pi)$ is a $C^q$- blid map on $Im(\pi)$, while the restriction $(H-\pi(H))|Ker(\pi)$ is a $C^q$-blid map on $Ker(\pi)$. Consequently, if $X_1\subset X$ is a subspace, such that there exists another subspace of $X$, so that these two form a complementary pair, then $X_1$ possesses a blid map.
 \end{itemize}

\begin{corollary}
Let  a space $X$ be as in items \textup{1-6.} Then for any Banach space $Y$ and any $C^q$-germ at zero there is a global $C^q$-representative. If the germ contains a local representative with bounded derivatives, then there is a global one with the same property.
\end{corollary}

\section{Applications}
\subsection{The Borel lemma for Banach spaces}\label{sec-Borel-lemma}
Let $X$ be a linear space over a field $\lowK$ ($char \ \lowK =0$) and $Y$ be a linear space over a field $\K$ ($\lowK\subset\K$). A map $P_j:X \to Y$ is called a polynomial homogeneous map of degree $j$ if there is a $j$-linear map
$$
                   g: 
\underbrace{
X\times X\times ...\times X 
}_{j}
\to Y
$$
such that $P_j(x)=g(x,x,...,x)$. 
For the given $P_j(x)$, the map $g$ is not unique, but there is a unique symmetric one. We will assume that $g$ is symmetric.
Then, the first derivative of $P_j$ at a point $z\in X$, is a linear map $X \to Y$, and can be calculated by the formula
 $$
          P'_j(z)(x)=jg(\underbrace{z,...,z}_{j-1},x)
 $$
In general, the derivative of order $n\leq j$, is a homogeneous polynomial map of degree $n$ and equals
 $$
          P_j^{(n)}(z)(x)^n=j(j-1)...(j-n+1)g(\underbrace{z,...,z}_{j-n},\underbrace{x,...x}_{n})
 $$
 In particular,
 $$
          P_j^{(j)}(z)(x)^j=j!g(x,...,x)=j!P_j(x)
 $$
does not depend on $z$. And lastly,  for $n>j$, $P_j^{(n)}(z)=0$.
 
It follows that all derivatives of $P_j$ at zero are zero, except for the order $j$.
 
The latter equals  $P_j^{(j)}(0)(x)^j=j!P_j(x)$.

Now, let $X$ and $Y$ be Banach spaces. Recall that $X$ must be real, while $Y$ can be real or complex. Let $f: X\to Y$ be  a local $C^\infty$ map. Then
 \begin{equation}\label{P-j}
                        P_j(x)=f^{(j)}(0)(x)^j
 \end{equation}
for any $j=0,1,...$ is a polynomial map of degree $j$, and it is continuous and even $C^{\infty}$. Therefore, for some $c_j>0$ we have the estimate
\begin{equation}\label{estimate}
||P_j(x)||\leq c_j ||x||^j ,\   \ x\in X.
\end{equation}
 
\begin{question}\label{surj} Given a sequence $\{P_j\}_{j=0}^{\infty}$ of continuous polynomial maps from $X$ to $Y$ of degree $j$, does there exist a $C^\infty$-germ $f: X \to Y$, which satisfies \eqref{P-j} for all $j=0,1,...$? \end{question}

The classical Borel lemma states, that given a sequence of real numbers $\{a_n\}$, there is a $C^\infty$ function $f$ on the real
line such that  $f^{(n)}(0)=a_n$. This means that the answer to the Question~\ref{surj} is affirmative for $X=Y=\R$. The same is true for finite-dimensional $X$ and $Y$.

\begin{theorem}[The Borel lemma] Let a Banach space X possess a $C^\infty$-blid map with bounded derivatives of all orders.
Then for any Banach space Y and any sequence $\{P_j\}_{j=0}^{\infty}$ of continuous homogeneous polynomial maps from $X$ to $Y$ there is a
$C^\infty$-map $f: X\to Y$ with bounded derivatives of all orders such that \eqref{P-j} is satisfied for all $j=0,1,...$
 \end{theorem}

\begin{proof} Let $H$ be a $C^\infty$-blid map at zero with bounded derivatives on $X$. Set $H_j(x)=\epsilon_j H(x/\epsilon_j)$. For a given $\epsilon_j$ the map $H_j(x)$ is a $C^{\infty}$-blid map also. 
 
Then the map $P_j(H_j(x))$ belongs to $C^\infty(X,Y)$, and all its derivatives at $0$ are zero, except for  the order $j$. The latter equals to $P_j^{(j)}(0)(x)^j=j!P_j(x)$. In addition, all derivatives of the map are bounded, and the derivative of order $n$ allows the following estimate
 $$
                        ||(P_j(H_j(x))^{(n)}||\leq \epsilon_j^{j-n}c_{j,n}
$$
with constants $c_{j,n}$ depending only on the maps $P_j$, $H$, and not depending on a choice of $\epsilon_j$. 
Therefore, under an appropriate choice of $\epsilon_j$ the series
 $$
     f(x)=\sum_0^\infty \frac{1}{j!} P_j\left(H_j(x)\right)
$$
converges in $C^\infty$ topology to a map from $X$ to $Y$. It is clear that
                      $$f^{(n)}(0)(x)^n=P_n(x).$$
This equality proves the statement.
\end{proof}
 
\begin{corollary} Let a space X  be as in items \textup{1-6} of Section~\textup{\ref{sec-examples}}. 
Then for any Banach space Y and any sequence $\{P_j\}_{j=0}^{\infty}$ of continuous homogeneous polynomial maps from $X$ to $Y$ there is a
$C^\infty$-map $f:X\to Y$ with bounded derivatives of all orders such that \eqref{P-j} is satisfied for all $j=0,1,...$
\end{corollary}

\subsection{Decomposition lemmas}\label{sec-decomp-lemmas}
In this section we will state lemmas that are useful in local and global analysis. We will use them in Section~\ref{section-equation}, where we present solution of the cohomological equation.

Let $X$ possesses a $C^\infty$-blid map $H$, and can be decomposed into a sum of two subspaces,
$$
          X=X_++X_-.
$$
Then there is a bounded projector $\pi_+:X \to X$ with $Im(\pi_+)=X_+$,
and the bounded projector $\pi_-=id-\pi_+$ onto $Ker(\pi_+)=X_-$. One can
write
$$ 
         x=(x_+, x_-),\  x\in X, \ x_{\pm}\in X_{\pm}.
$$
Denote by $H_+(x_+)$ and $H_-(x_-)$ the corresponding $C^\infty$-blid maps on $X_+$ and $X_-$.

      Let now $Y$ be a Banach space. Recall that a $C^\infty$-map $f:X \to Y$ is called flat on a subset $S$ if it vanishes on $S$ together with all its derivatives. 

\begin{lemma}\label{f-plus-minus} Let all derivatives of $H$ be bounded on $X$. Let a map $f_0:X \to Y$ be bounded with all derivatives on every bounded subset, and be flat at zero. Then there is a decomposition $f_0=f_++f_-$ and a neighborhood $U$ of zero such that
$f_{+}$ ($f_{-}$) has the same boundedness property and is flat on the intersection $X_{+}\cap U$ ($X_{-}\cap U$). 
\end{lemma}
\begin{proof}Let $\epsilon_j<1$ and
$$
 P_j(x)=\left(\frac{\partial^j}{\partial x_-^j} f_0\left(H(x\right))_{|x_-=0}\right)\left(\epsilon_jH_-(x_-/\epsilon_j)\right)^j.
$$
Then $P_j$ are flat on $X_+$ in a neighborhood $$U=\{x: H(x)=x\},$$ and
\begin{equation}\label{partial-der1}
\left(\frac{\partial^{p+q}}{\partial x^p_+ \partial x^q_-} \left(P_j(x)\right)\right)|_{x_-=0}=0,\ \  q\neq j,\end{equation}
while 
\begin{equation}\label{partial-der2}
 \left(   \frac{\partial^{p+q}}{\partial x^p_+ \partial x^q_-} \left(P_j(x)\right)\right)|_{x_-=0}=j!\left(\frac{\partial^{p+q}}{\partial x^p_+ \partial x^q_-} f_0(H(x))\right)|_{x_-=0},\ \  q= j.
\end{equation}
Additionally, the maps satisfy an estimate
$$
  ||P_j^{(n)}(x)||\leq c_{n,j}\epsilon_j^{j-n}.
$$
Therefore, the series
$$
    f_+(x)=\sum_0^\infty \frac{1}{j!}P_j(x)
$$
converges in the $C^\infty$-topology under an appropriate choice of $\epsilon_j$. Its sum $f_+$ is a $C^\infty$-map from $X \to Y$ flat on $X_+$. Equations \eqref{partial-der1} and \eqref{partial-der2} imply
$$
         f_+^{(n)}(x)|_{x_-=0}=\left(\left(f_0(H(x))\right)^{(n)}\right)|_{x_-=0}
$$
As a result, the map
$$
        f_-(x)=f_0(x)-f_+(x) 
$$
 is flat on $X_-$ when $H(x)=x$.
\end{proof}

\begin{lemma}\label{v-plus-minus} Let a $C^\infty$-map $v:X \to Y$ vanish in a neighborhood of zero. Then there is a decomposition
$$
             v=v_++v_-
$$
into a sum of maps vanishing on strips $S_+(\epsilon)=\{x: ||x_+||<\epsilon\}$ and $S_-(\epsilon)=\{x: ||x_-||<\epsilon\}$ respectively.
\end{lemma}

\begin{proof} Let $v(x)=0$ as $||x||<\delta$. One can choose $\epsilon<\delta$ and a $C^\infty$-blid map $H$ to be such that
$$  ||H(x)||<\delta,\ \  x\in X,\ \  H(x)=x \mbox{  as  } ||x||<\epsilon.$$

Then the $C^\infty$-map
$$
          v_+(x)=v(x_+,H_-(x_-))
$$
vanishes as $||x_+||<\epsilon$, while the map
$$
       v_-(x)=v(x)-v_+(x)
$$
vanishes as $||x_-||<\epsilon.$
\end{proof}

Note that in the Lemma~\ref{v-plus-minus} we do not assume boundedness of derivatives of map $H$. 

\subsection{Cohomological equations in Banach spaces}\label{section-equation}

Given a map $ F:X \to X$, the equation 
 \begin{equation}\label{cohomol-eqn}
g(Fx)- g(x)=f(x)
 \end{equation}
will be called cohomological equation with respect to the unknown $C^\infty$ function $g: X \to \C$. 

Various versions of this equation are well-known and have been studied in multiple articles. Yu.I. Lyubich in the article~\cite{L} presents a very broad overview of this area and considers \eqref{cohomol-eqn} in a non-smooth category. For a discussion of smooth cohomological equations we recommend the book \cite{B-T}. 

In the present work, we consider a linear space $X$ over a field $K$, and a linear map $F=A: X \to X$. These cohomological equations are often studied in the theory of normal forms.

Consider a homogeneous polynomial map $f(x)=P_n(x)$ of degree $n>0$ and linear map $F=A$. Let us also look for a solution in a polynomial form, $g(x)=Q_n(x)$.
Then we arrive at the equality
 \begin{equation}\label{equality}
    (L_n-id)Q_n(x)=P_n(x),
 \end{equation}
where $L_nQ_n(x)=Q_n(Ax)$. If the operator $L_n-id$ is
invertible, then \eqref{equality} has a solution $Q_n(x)$ for every
$P_n(x)$. Otherwise, additional restrictions on $P_n(x)$ arise. If $X=K^m$ is finite-dimensional, then the invertability of $L_n-id$ is provided by the absence of  the resonance relations 
 
 \begin{equation}\label{res}
 1=a_1^{p_1}a_2^{p_2}....a_m^{p_m},
 \end{equation}
$a_k\in spec A,$ $\sum_1^m p_i =n.$
 
Therefore, the absence of resonance relations for all $n>0$
ensures the solvability of \eqref{equality} for any $n>0$ and any $P_n.$ If $K=R$, then the mentioned condition implies also that $A$ is hyperbolic, i.e., its spectrum does not intersect the unit circle in $\C$.
  Consider now \eqref{cohomol-eqn} with real finite-dimensional space 
$X$, and $f\in C^\infty$. If all equations \eqref{equality} are solvable for $P_n(x)=f^{(n)}(0)(x)^n$, then \eqref{cohomol-eqn} is called formally
solvable at zero. It is known that if $A$ is hyperbolic, invertible authomorphism, then formal solvability implies $C^\infty$ solvability. Our aim is to prove a similar assertion in the infinite-dimensional case.

     So, let $A:X \to X$ be a continuous invertible hyperbolic linear operator.
 Then, there is a direct decomposition $X=X_+ + X_-$ in a sum of $A$-invariant
subspaces such that spec $A_+=specA|X_+$
lies inside of the circle, while spec $A_-=spec A|X_-$ lies outside. Moreover, there is an equivalent norm in $X$ such that
 
 \begin{equation}\label{contraction}    
\begin{array}{l}
||A_+x||<q||x_+||,
\\
||A_-^{-1}x||<q||x_-||
\end{array}
 \end{equation}
with some $q<1.$
  If  \eqref{cohomol-eqn} has a solution, then $f(0)=0$, and we assume this condition to be fulfilled.
 
\begin{proposition} Let one of the subspaces $X_+$, $X_-$ be trivial, then \eqref{cohomol-eqn} has a global $C^\infty$-solution for any $C^\infty$-function $f$,
$f(0)=0$ with bounded derivatives on every bounded subset $S$.
 \end{proposition}

\begin{proof} Assume $X_-=0$. Then the series
 $$
          g(x)=-\sum_0^\infty f(A^k x)
 $$
converges in the space of $C^\infty$-functions, since
 $$
                  ||(f(A^kx))^{(n)}||\leq q^{kn} c_n(S)
 $$
for every bounded subset $S\subset X=X_+$.
 
Similarly, if $X_+=0$, then the series
 $$
            g(x)=\sum_1^\infty f(A^{-k}x)
 $$
leads to a solution we need.
 \end{proof}
   If both of the spaces $X_+,X_-$ are non-trivial, then the existence of solutions requires
additional assumptions and constructions.
As above, $P_n(x)=f^{(n)}(0)(x)^n$ is a continuous homogeneous polynomial of degree $n$.
Let $g$ be a solution of \eqref{cohomol-eqn}. Differentiating both of its sides, we arrive at equation
\begin{subequations}
\label{eqns}
\begin{equation}
A_n Q_n(x)-Q_n(x)=P_n(x),\  \  n=1,2,3....,
\tag{\ref*{eqns}n}  \label{n-eqns}
\end{equation}
\end{subequations}
where $Q_n(x)=g^{(n)}(0)(x)^n$, and $A_n$ is a linear map in the Banach space of continuous homogeneous polynomials. It arises after $n$-multiple differentiation of the function $g(Ax)$. 
If all equations \eqref{n-eqns} have continuous solutions, then we say that \eqref{cohomol-eqn} is formally solvable at zero. The latter a-priory takes place if identity does not belong to $spec A_n$. The opposite condition  can be considered as an  infinite-dimensional version of resonance relations.
 
\begin{theorem} Let $A$  be a hyperbolic linear automorphism, and $X$ possesses a $C^\infty$-blid map with bounded derivatives on $X$. If all derivatives of f are bounded on every bounded subset, and \eqref{cohomol-eqn} is formally
solvable at zero (i.e. all equations \eqref{n-eqns} have continuous solutions), then there exists a global $C^\infty$-solution $g(x)$.
 \end{theorem}
 
\begin{proof}First, we will show that a formal solvability implies a local one, i.e., implies existence of (a global) $C^\infty$-function $g$ such that \eqref{cohomol-eqn} holds in a neighborhood of zero. 

Let $Q_n(x)$ be continuous solutions of~\eqref{n-eqns}. We will build a $C^\infty$-function $g_0(x)$, following the Borel lemma. The substitution $g=w+g_0$ reduces \eqref{cohomol-eqn} to the equation
 \begin{equation}\label{flat-eqn}
    w(Ax)-w(x)=f_0(x),
 \end{equation}
where $f_0(x)=f(x)-g_0(Ax)+g_0(x)$ is flat at zero. Lemma~\ref{f-plus-minus} provides the decomposition
$$f_0(x)=f_+(x)+f_-(x)$$ 
together with a neighborhood $U$.
Let $H$ be a blid map. One can assume that $H(X)\subset U$. Consider the pair of equations:
\begin{equation}\label{flat-plus-minus-eqns}
\begin{array}{l}
            w_+(Ax)-w_+(x)=f_+(H(x)),\\
            w_-(Ax)-w_-(x)= f_-(H(x)).
\end{array}
\end{equation}
Then, the estimates \eqref{contraction}  imply that the series 
$$\
\begin{array}{l}
   w_+(x)=-\sum_0^\infty f_+(H(A^kx)) \mbox{,  }
  w_-(x)=\sum_1^\infty f_-(H(A^{-k}x))
\end{array}
$$
converge since $f_{\pm}$ is flat on $X_{\pm}$. The series present solutions to the first and second equations \eqref{flat-plus-minus-eqns} correspondingly.
 
 The (global) function $w_1(x)=w_-(x)+w_+(x)$ satisfies \eqref{flat-eqn} in
a neighborhood of the origin, i.e. it is local solution for this equation. In turn, the function $g=w_1+g_0$ is a
local solution of the initial equation \eqref{cohomol-eqn}.
 
Now we will prove that local solvability implies a
global one. Let $\gamma_0$ be a local solution. The substitution $g=h+\gamma_0$
reduces \eqref{cohomol-eqn} to the equation
 \begin{equation}\label{flat-eqn-for-global-sol}
   h(Ax)-h(x)=v(x)
\end{equation}
where $v(x)=f(x)-\gamma(Ax)+\gamma(x)$ vanishes in a neighborhood of zero.
Let $v=v_++v_-$ be a decomposition described by Lemma~\ref{v-plus-minus}. Consider the series 
\begin{equation}\label{series-global-sol}
    h_+(x)=-\sum_0^\infty v_+(A^kx).
\end{equation}
Since $||A_+||<1$, for any bounded set $D$ there is a number $k_0(D)$
such that $A^k(D)\subset S_+(\epsilon)$ for $k>k_0(D)$. Hence, if $k> k_0(D)$, then $v_+(A^kx)=0$ for all $x\in D$, and \eqref{series-global-sol} represents a global smooth function. By the same arguments the series
$$          h_-(x)=\sum_1^\infty v_-(A^{-k}x)$$
is a global smooth function. The sum $h=h_++h_-$ is a smooth global solution of \eqref{flat-eqn-for-global-sol}. This construction completes the proof.
\end{proof}

\section{More examples and open questions}\label{sec-questions}

One of the important questions of extensions of local maps on Banach spaces is the following. Do Banach spaces without smooth blid maps  exist? Recently, affirmative answer was presented in \cite{DH} (also see \cite{HJ}). The authors proved that there exist Banach spaces that do not allow $C^2$-extension (and hence the $C^2$-blid map).
\begin{question}\label{Q-arbitrary-space} For which spaces do smooth blid maps exist? 
Do they exist on $l_p$, with non-even $p$?
\end{question}

Some other open questions related to the generalization of Theorem~\ref{thm-extension} we discuss below.

How can we extend germs of maps defined at a closed subset $S\subset X$? For this construction we need to define smooth blid maps at $S$. More precisely, generalizing the definition of germs at a point, we will say that maps $f_1$ and $f_2$ from neighborhoods $U_1$ and $U_2$ of $S$ into $Y$ are {\it equivalent}, if they coincide in a (smaller) neighborhood of $S$.
Every equivalence class is called a {\it germ at $S$}. We pose the same question. Given a $C^q$-germ at $S$, does there exist a global representative? Assume there exist a $C^q$-map $H:X \to X$ whose image $H(X)$ is contained in a neighborhood  $U$ of $S$
and which is equal to the identity map in a smaller neighborhood. Such maps we call {\it smooth blid maps at S}. Then every local map $f$ defined in $U$ can be extended on the whole $X$.  It suffices to set $F(x)=f(H(x))$.

In the next example, we construct the map $H$ for a segment (in particular, for a ball).
\begin{example}\label{example-segment}
Let $S(A)$ be a set of all functions $x\in C[0,1]$ whose graphs $(t,x(t))$ are
contained in a closed $A\subset\R^2$, where $A$ is chosen in such a way that $S(A)\neq \emptyset$. Let  $h(t,x)$ be a $C^\infty$-function, which is equals to $1$ in a neighborhood of $A$ and vanishes 
outside of a bigger set. Then, for an arbitrary $y\in C[0,1]$
$$
             H_y(x)(t)=y(t)+h(t,x(t))(x(t)-y(t))
$$
is a $C^{\infty}$-blid map for $S(A)$.

If $A=\left\{\{t,x\}: \min(\psi(t),\phi(t))\leq x\leq  \max(\psi(t),\phi(t))\right\}$ for some $\phi,\psi \in C[0,1]$, then $S(A)$ can be thought of as a segment $[\phi,\psi]\subset C[0,1]$. 

In particular, given $z\in C[0,1]$ and a constant $r>0$, setting $\phi =z-r$ and $\psi=z+r$, we obtain the ball $B_r(z)=\{ x: ||x-z||\leq r \}\subset C[1,0]$.

Every $C^q$-germ at $[\phi,\psi]\subset C[0,1]$ contains a global representative. 

Note, this example has an obvious generalization to segments and balls in $C^k[0,1]$.
\end{example}

The Question \ref{Q-arbitrary-space} and Example \ref{example-segment} bring us to the next question. 
\begin{question} For which pairs $(S,X)$ do similar constructions exist? In particular, can a smooth blid map be constructed for any bounded subset $S$ of a space $X$ possessing a smooth blid map? For example, we do not know whether a smooth blid map can be constructed for a sphere $S=\{x\in C[0,1]: ||x||=r\}$. 
\end{question}
\begin{question} The Borel lemma for finite-dimensional spaces is a particular case of the well-known Whitney extension theorem from a closed set $S\subset \R^n$. What is an infinite-dimensional version of the Whitney theorem? 
\end{question}

\bibliographystyle{amsalpha}

\begin{thebibliography}{A}
\bibitem[B]{B} G. Belitskii, \textit{The Sternberg theorem for a Banach space,} Funct. Anal. Appl., \textbf{18} (1984), 238--239. MR 0757253 (86b:58097), \href{http://link.springer.com/article/10.1007\%2FBF01086163}{doi:10.1007/BF01086163}.
\bibitem[B-T]{B-T} G. Belitskii, V. Tkachenko, One-dimensional functional equations, \textit{Oper. Th.: Adv. and Appl., 144, Birkh{\"a}auser Verlag,} 2003.
\bibitem[B-R]{B-R} G. Belitskii, V. Rayskin, \textit{Equivalence of families of diffeomorphisms on Banach spaces}, Math. preprint archive, UT Austin, 07-71. \href{https://www.ma.utexas.edu/mp_arc-bin/mpa?yn=07-71}{https://www.ma.utexas.edu/mp\underline{ }arc-bin/mpa?yn=07-71}.
\bibitem[C]{C} H. Cartan, Calcul Diff\'erentiel, \textit{Hermann, Paris,} 1967 178 pp. MR 0223194 (36 \#6243)
\bibitem[DH]{DH} S. D'Alessandro and P. Hajek, \textit{Polynomial algebras and smooth functions in Banach spaces}, Journal of Functional Analysis \textbf{266} (2014), 1627-1646.
\bibitem[F-M]{F-M} R. Fry, S. McManus, \textit{Smooth bump functions and the geometry of banach spaces: A brief survey}, Expositiones Mathematicae, \textbf{20(2)} (2002) 143--183, \href{https://doi.org/10.1016/S0723-0869(02)80017-2}{https://doi.org/10.1016/S0723-0869(02)80017-2}
\bibitem[HJ]{HJ} P. Hajek and M. Johanis, Smooth Analysis in Banach Spaces, \textit{Walter de Gruyter, GmbH, Berlin}, 2014 497 pp.
\bibitem[M]{M} V.Z. Meshkov, \textit{Smoothness properties in Banach spaces,} Studia Mathematica, \textbf{63} (1978), 111--123. MR 0511298 (80b:46027), \href{http://matwbn.icm.edu.pl/ksiazki/sm/sm63/sm6319.pdf}{http://matwbn.icm.edu.pl/ksiazki/sm/sm63/sm6319.pdf}.
\bibitem[K]{K} J. Kurzweil, \textit{On approximation in real Banach spaces,} Studia Math., \textbf{14.2,} (1954), 213--231. MR 0068732 (16,932g), \href{http://eudml.org/doc/216828}{http://eudml.org/doc/216828}.
\bibitem[L]{L} Yu. Lyubich, \textit{The cohomological equations in nonsmooth categories}, arXiv:1211.0229v1 [math. FA] 1 Nov.2012, \href{https://arxiv.org/pdf/1211.0229.pdf}{https://arxiv.org/pdf/1211.0229.pdf}.
\bibitem[N]{N} Z. Nitecki, Differentiable Dynamics. An Introduction to the Orbit Structure of Diffeomorphisms, \textit{The MIT Press, Camridge, Mass.-London}, 1971. xv+282 pp. MR 0649788 (58 \#31210)
\bibitem[R]{R} V. Rayskin, \textit{Theorem of Sternberg-Chen modulo central manifold for Banach spaces,} Ergodic Theory \& Dynamical Systems, \textbf{29} (2009), \textit{no. 6},1965--1978. MR 2563100 (2011a:37038), \href{https://doi.org/10.1017/S0143385708000989}{https://doi.org/10.1017/S0143385708000989}
\bibitem[P]{P} J. Palis, \textit{Local Structure of Hyperbolic Fixed Points in Banach Space},
Anais da Academia Brasileira de Ciencias, \textbf{40} (1968), 263--266.
\bibitem[S]{S} S. Soboleff \textit{Sur un th{\'e}or{\`e}me d'analyse fonctionnelle}, Rec. Math. [Mat. Sbornik] N.S., \textbf{4(46):3} (1938), 471--497.
\bibitem[St]{St} S. Sternberg, \textit{On the structure if local homeomorphisms of Euclidean n-space II}, Amer. J. Math., \textbf{80} (1958), 623--631.

\end{thebibliography}

\end{document}